\documentclass[12pt]{amsart}

\usepackage{bbm,bm}
\usepackage{amssymb}

\usepackage{color}

	\definecolor{mygreen}{rgb}{.1,.5,.1}
	\definecolor{myblue}{rgb}{.1,.1,.5}
	\definecolor{mymagenta}{cmyk}{0,.4,0,0}

\usepackage[bookmarks=false,colorlinks,linkcolor=myblue,citecolor=mygreen]{hyperref}

\date{18 August 2011}

\theoremstyle{plain}
\newtheorem{theorem}{Theorem}
\newtheorem{corollary}[theorem]{Corollary}
\newtheorem{lemma}[theorem]{Lemma}
\newtheorem{proposition}[theorem]{Proposition}

\theoremstyle{definition}
\newtheorem{example}[theorem]{Example} 

\theoremstyle{remark}
\newtheorem*{remark}{Remark} 

\newcommand{\demph}[1]{\textcolor{myblue}{\it #1}}

\def\id{\mathrm{id}}
\newcommand{\coker}{\mathrm{coker}}
\newcommand{\hker}{\mathrm{Hker}}
\newcommand{\lker}{\mathrm{Lker}}
\newcommand{\eq}{\mathrm{Eq}}
\DeclareMathOperator{\supp}{supp}

\newcommand{\abs}[1]{\lvert#1\rvert} 

\def\field{\Bbbk}

\let\onto=\twoheadrightarrow
\let\map=\xrightarrow

\newcommand{\exponential}[2]{\mathcal E_{#1}(#2)}
\newcommand{\type}[2]{\mathcal T_{#1}(#2)}
\newcommand{\cycle}[2]{\mathcal Z_{#1}(#2)}
\newcommand{\poincare}[2]{\mathcal O_{#1}(#2)} 

\newcommand{\tone}{\mathbf{1}}

\newcommand{\rH}{\mathrm{H}}

\newcommand{\bh}{\mathbf h}
\newcommand{\bl}{\mathbf L}
\newcommand{\be}{\mathbf E}
\newcommand{\bk}{\mathbf k}
\newcommand{\bc}{\mathbf c}
\newcommand{\bp}{\mathbf p}
\newcommand{\bq}{\mathbf q}
\newcommand{\bmm}{\mathbf m}
\newcommand{\bu}{\mathbf u}
\newcommand{\bv}{\mathbf v}

\newcommand{\bX}{\mathbf{X}}
\newcommand{\bLie}{\mathbf{Lie}}
\newcommand{\bel}{\mathbf{e}} 

\newcommand{\Pc}{\mathcal{P}} 
\newcommand{\Sc}{\mathcal{S}} 
\newcommand{\Uc}{\mathcal{U}} 

\newcommand{\bN}{\mathbb{N}}
\newcommand{\bQ}{\mathbb{Q}}
\newcommand{\bZ}{\mathbb{Z}}

\newcommand{\K}{\mathcal K}
\newcommand{\Kb}{\overline{\K}}

\newcommand{\sym}{\textsl{Sym}}
\newcommand{\qsym}{\textsl{QSym}}

\newcommand{\bdot}{\bm\cdot}

\def\pmrg{{\textcolor{mymagenta}{.}}}
\def\cmrg{{\textcolor{mymagenta}{|}}}

\def\llb{{[\![}}
\def\rrb{{]\!]}}

\newcommand{\qand}{\quad\text{and}\quad}

\newcommand{\djcup}{\sqcup}

\author{Marcelo Aguiar}
\address[Aguiar]{
	Department of Mathematics\\
        Texas A\&M University\\
        College Station, TX\, 77843 
        }
\email{maguiar@math.tamu.edu}
\urladdr{http://www.math.tamu.edu/$\small\sim$maguiar}
\thanks{Aguiar supported in part by NSF grant DMS-1001935.}

\author{Aaron Lauve}
\address[Lauve]{
	Department of Mathematics and Statistics \\
	Loyola University Chicago \\
	Chicago, IL\, 60660 
        }
\email{lauve@math.luc.edu}
\urladdr{http://www.math.luc.edu/$\small\sim$lauve}
\thanks{Lauve supported in part by NSA grant H98230-11-1-0185.}

\title{{L}agrange's Theorem for {H}opf Monoids in Species}



\begin{document}

\begin{abstract}
Following Radford's proof of Lagrange's theorem for pointed Hopf algebras,
we prove Lagrange's theorem for Hopf monoids in the category of connected species. 
As a corollary, we obtain necessary conditions for a given subspecies $\bk$ of a Hopf monoid $\bh$ to be a Hopf submonoid: the quotient of any one of the generating series of $\bh$ by the corresponding generating series of $\bk$ must have nonnegative coefficients. Other corollaries include a necessary condition for a sequence of nonnegative integers to be the
dimension sequence of a Hopf monoid 
in the form of certain polynomial inequalities, and of
a set-theoretic Hopf monoid in the form of certain linear inequalities.
The latter express that the binomial transform of the sequence must be nonnegative. 
\end{abstract}

\maketitle

\section*{Introduction}

Lagrange's theorem states that for any subgroup $K$ of a group $H$, 
$H\cong K\times Q$ as (left) $K$-sets, where $Q=H/K$. In particular,
if $H$ is finite, then $|K|$ divides $|H|$. 
Passing to group algebras over a field $\field$, we have that 
$\field H \cong \field K \otimes \field Q$ as (left) $\field K$-modules, or that $\field H$ is free as a $\field K$-module.  Kaplansky~\cite{Kap:1975} conjectured that the same statement holds for Hopf algebras---group algebras being principal examples. 
It turns out that the result does not hold in general, as shown by Oberst and Schneider~\cite[Proposition~10]{ObeSch:1974} and \cite[Example~3.5.2]{Mon:1993}.
On the other hand, the result does hold for certain large classes of Hopf algebras,
including the finite dimensional ones by a theorem of
Nichols and Zoeller~\cite{NicZoe:1989}, and the pointed ones by a theorem of Radford~\cite{Rad:1977}. Further (and finer) results of this nature were developed by Schneider~\cite{Sch:1990,Sch:1992}. Additional work on the conjecture includes that of
Masuoka~\cite{Mas:1992} and Takeuchi~\cite{Tak:1979}; more information can be
found in Sommerh\"auser's survey~\cite{Som:2000}.

The main result of this paper (Theorem~\ref{t:main}) is a version of Lagrange's theorem for Hopf monoids in the category of connected species: if $\bh$ is a connected Hopf monoid
and $\bk$ is a Hopf submonoid, there exists a species $\bq$ such that $\bh=\bk\bdot\bq$.
 An immediate application is a test for Hopf submonoids (Corollary~\ref{c:Sp series}): if any one of the generating series for a species $\bk$ does not divide in $\bQ_{\geq 0}\llb x\rrb$ the corresponding generating series for the Hopf monoid $\bh$ (in the sense that the quotient has at least one negative coefficient), then $\bk$ is not a Hopf submonoid of $\bh$. A similar test also holds for connected graded Hopf algebras (Corollary~\ref{c:cgVec series}). The proof of Theorem~\ref{t:main}
for Hopf monoids in species parallels Radford's proof for Hopf algebras.
(Hopf algebras are Hopf monoids in the category of vector spaces).

The paper is organized as follows. In Section~\ref{s:Hopf algebras}, 
we recall Lagrange's theorem for Hopf algebras,
focusing on the case of connected graded Hopf algebras.
In Section~\ref{s:Hopf monoids},
we recall the basics of Hopf monoids in species and prove Lagrange's theorem in
this setting. Examples and applications are given in Section~\ref{s:applications}. 
Among these, we derive certain polynomial inequalities that a sequence of nonnegative integers must satisfy in order to be the dimension sequence of a connected Hopf monoid in species. 
In the case of a set-theoretic Hopf monoid structure, we obtain additional necessary conditions in the form of linear inequalities which express that the binomial transform of the enumerating sequence must be nonnegative. 
In Section~\ref{s:dimensions} we provide information on the growth and support of the dimension sequence of a connected Hopf monoid. The latter must be an additive submonoid of the natural numbers.
We conclude in Section~\ref{s:kernels}
with information on the species $\bq$ entering in Lagrange's theorem. In the dual setting,
$\bq$ is the Hopf kernel of a morphism, and for cocommutative Hopf monoids it can
be described in terms of Lie kernels and primitive elements via the Poincar\'e-Birkhoff-Witt
theorem.

All vector spaces are over a fixed field $\field$ of characteristic $0$, except in Section~\ref{s:dimensions}, where the characteristic is arbitrary.

\section{Lagrange's theorem for Hopf algebras}\label{s:Hopf algebras}
We begin by recalling a couple of
versions of this theorem. 

\begin{theorem} Let $H$ be a finite dimensional Hopf algebra over a field $\field$. If $K\subseteq H$ is any Hopf subalgebra, then $H$ is a free left (and right) $K$-module.
\end{theorem}
This is the Nichols-Zoeller theorem \cite{NicZoe:1989}; see also \cite[Theorem~3.1.5]{Mon:1993}. We will not make direct use of this result, but instead of the related results discussed below.

A Hopf algebra $H$ is \demph{pointed} if all its simple subcoalgebras are $1$-dimensional. Equivalently, the group-like elements 
of $H$ linearly span the coradical of $H$.

Given a subspace $K$ of $H$, set 
\[
K_+:=K\cap\ker(\epsilon),
\]
where $\epsilon:H\to\field$ is the counit of $H$. Let $K_+H$ denote the right $H$-ideal
generated by $K_+$.

\begin{theorem}\label{t:pointed} 
Let $H$ be a pointed Hopf algebra. If $K\subseteq H$ is any Hopf subalgebra, then $H$ is a free left (and right) $K$-module. Moreover,
\[
H\cong K\otimes (H/K_+H)
\]
as left $K$-modules.
\end{theorem}
The first statement is due to Radford \cite[Section~4]{Rad:1977} and the second (stronger) statement to Schneider \cite[Remark~4.14]{Sch:1990}, \cite[Corollary~4.3]{Sch:1992}.
Various generalizations can be found in these references as well as in 
Masuoka~\cite{Mas:1992} and Takeuchi \cite{Tak:1979}; see also Sommerha\"user~\cite{Som:2000}.
 We are interested in the particular variant given in Theorem~\ref{t:connected} below.

A Hopf algebra $H$ is \demph{graded} if there is given a decomposition
\[
H=\bigoplus_{n\geq0} H_n
\]
into linear subspaces that is preserved by all operations.  It is \demph{connected graded} if
in addition $H_0$ is linearly spanned by the unit element.  

\begin{theorem}\label{t:connected} 
Let $H$ be a connected graded Hopf algebra. If $K\subseteq H$ is a graded Hopf subalgebra, then $H$ is a free left (and right) $K$-module. Moreover,
\[
H\cong K\otimes (H/K_+H)
\]
as left $K$-modules and as graded vector spaces.
\end{theorem}
\begin{proof}
Since $H$ is connected graded, its coradical is $H_0 = \field$, so $H$ is pointed and Theorem~\ref{t:pointed} applies. Radford's proof shows that there exists a graded
vector space $Q$ such that
\[
H \cong K \otimes Q
\]
as left $K$-modules and as graded vector spaces. (The argument we give in the parallel setting of Theorem~\ref{t:main} makes this clear.)
Note that $K_{+\!}=\bigoplus_{n\geq1} K_n$, hence $K_+H$ and $H/K_+H$ inherit the grading of $H$.
To complete the proof, it suffices to show that $Q\cong H/K_+H$ as graded vector spaces. 

Let $\varphi:K\otimes Q\to H$ be an isomorphism of left $K$-modules and of graded vector spaces. We claim that
\[
\varphi(K_+\otimes Q)=K_+H.
\]
In fact, since $\varphi$ is a morphism of left $K$-modules,
\[
\varphi(K_+\otimes Q) = K_+\varphi(1\otimes Q)\subseteq K_+H.
\]
Conversely, if $k\in K_+$ and $h\in H$, writing $h=\sum_i \varphi(k_i\otimes q_i)$ with $k_i\in K$ and $q_i\in Q$, we obtain
\[
kh=\sum_i \varphi(kk_i\otimes q_i)\in \varphi(K_+\otimes Q),
\]
since $K_+$ is an ideal of $K$. 

Now, since $K=K_0\oplus K_+$, we have
\[
K\otimes Q =(K_0\otimes Q)\oplus (K_+\otimes Q)
\]
and therefore
\[
H/K_+H = \varphi(K\otimes Q)/\varphi(K_+\otimes Q) \cong \varphi (K_0\otimes Q)\cong Q
\]
as graded vector spaces.
\end{proof}

Given a graded Hopf algebra $H$, let $\poincare{H}{x}\in \bN\llb x\rrb$ denote its
 \demph{Poincar\'e series}---the ordinary generating function for the sequence of dimensions of its graded components,
\begin{gather*}
	\poincare{H}{x} := \sum_{n\geq0} \dim H_n \, x^n .
\end{gather*} 
Suppose $H$ is connected graded and $K$ is a graded Hopf subalgebra.
In this case, their Poincar\'e series are of the form
\[
1+a_1x+a_2x^2+\cdots
\]
with $a_i\in\bN$ and the quotient $\poincare{H}{x}/\poincare{K}{x}$ is a well-defined power series in $\bZ\llb x\rrb$.

\begin{corollary}\label{c:cgVec series} 
Let $H$ be a connected graded Hopf algebra. If $K\subseteq H$ is any graded Hopf subalgebra, then the quotient $\poincare{H}{x} / \poincare{K}{x}$ of Poincar\'e series is nonnegative, i.e., belongs to $\bN\llb x\rrb$.
\end{corollary}
\begin{proof}
By Theorem~\ref{t:connected}, $H \cong K \otimes Q$ as graded vector spaces, where
$Q=H/K_+H$. Hence $\poincare{H}{x} = \poincare{K}{x} \, \poincare{Q}{x}$ and the result follows.
\end{proof}

\begin{example}
Consider the Hopf algebra $\qsym$ of quasisymmetric functions in countably many variables, and the Hopf subalgebra $\sym$ of symmetric functions. 
They are connected graded, so by Theorem~\ref{t:connected}, $\qsym$ is a free module over $\sym$. Garsia and Wallach
prove this same fact for the algebras $\qsym_n$ and $\sym_n$ of (quasi) symmetric functions in $n$ variables (where $n$ is a finite number)~\cite{GarWal:2003}. 
While $\qsym_n$ and $\sym_n$
are quotient algebras of $\qsym$ and $\sym$, they are not quotient coalgebras.
Since a Hopf algebra structure is needed in order to apply 
Theorem~\ref{t:connected}, we cannot derive the result of Garsia and Wallach in this manner.
The papers~\cite{GarWal:2003} and~\cite{LauMas:2011} provide information on the
space $Q_n$ entering in the decomposition $\qsym_n\cong\sym_n\otimes Q_n$.
\end{example}


\section{Lagrange's theorem for Hopf monoids in species}\label{s:Hopf monoids}

We first review the notion of Hopf monoid in the category of species, following \cite{AguMah:2010}, and then prove Lagrange's theorem in this setting.
We restrict attention to the case of connected Hopf monoids.

\subsection{Hopf monoids in species}\label{ss:species}

The notion of species was introduced by Joyal 
\cite{Joy:1981}. It formalizes the notion of combinatorial structure
and provides a framework for studying the generating functions which enumerate these structures. The book \cite{BerLabLer:1998} by Bergeron, Labelle and Leroux expounds the theory of set species.

Joyal's work indicates that species may also be regarded as algebraic objects; this is the point of view adopted in
\cite{AguMah:2010} and in this work. To this end, it is convenient to work with vector species.

A \demph{(vector) species} is a functor $\bq$ from {finite sets} and bijections to {vector spaces} and linear maps. Specifically, it is a family of vector spaces $\bq[I]$, one for each finite set $I$, together with linear maps $\bq[\sigma]: \bq[I] \to \bq[J]$, one for each bijection $\sigma:I\to J$, satisfying 
\[
	 \bq[\id_I] = \id_{\bq[I]} \quad\hbox{and}\quad 
	 \bq[\sigma\circ\tau] = \bq[\sigma]\circ \bq[\tau]
\]
whenever $\sigma$ and $\tau$ are composable bijections.

The notation $\bq[a,b,c,\ldots]$ is shorthand for $\bq[\{a,b,c,\ldots\}]$ and
$\bq[n]$ is shorthand for $\bq[1,2,\ldots,n]$. The space $\bq[n]$ is an $S_n$-module
via $\sigma\cdot v = \bq[\sigma](v)$ for $v\in\bq[n]$ and $\sigma\in S_n$.

A species $\bq$ is \demph{finite dimensional} if each vector space $\bq[I]$ is finite dimensional. 
In this paper, all species are finite dimensional. 

A morphism of species is a natural transformation of functors. Let ${\sf Sp}$
denote the category of (finite dimensional) species.

We give two elementary examples that will be useful later. 

\begin{example}\label{ex: example species}
Let $\be$ be the \demph{exponential species}, defined by $\be[I] = \field\{\ast_I\}$ for all $I$. 
The symbol $\ast_I$ denotes an element canonically associated to the set $I$ (for definiteness, we may take $\ast_I=I$). Thus, $\be[I]$ is a $1$-dimensional space 
with a distinguished basis element. A richer example is provided by the species $\bl$ of \demph{linear orders}, defined by  $\bl[I] = \field\{\hbox{linear orders on }I\}$ for all $I$ (a space of dimension $n!$ when $|I|=n$).
\end{example}

We use $\bdot$ to denote the \demph{Cauchy product} of two species. Specifically, 
\[
\bigl(\bp \bdot \bq\bigr)[I] := \bigoplus_{S\djcup T = I} \bp[S] \otimes \bq[T]
\quad\hbox{for all finite sets }I.
\]
The notation $S\djcup T = I$ indicates that $S\cup T = I$ and $S \cap T = \emptyset$.
The sum runs over all such \demph{ordered decompositions} of $I$, or equivalently over all
subsets $S$ of $I$: there is one term for $S\djcup T$ and another for $T\djcup S$. 
The Cauchy product turns ${\sf Sp}$ into a symmetric monoidal category. 
The braiding simply switches the tensor factors.
The unit object
is the species $\tone$ defined by
\[
\tone[I]  := \begin{cases}
\field & \text{if $I$ is empty,} \\
0 & \text{otherwise.}
\end{cases}
\]

A \demph{monoid} in the category $({\sf Sp},\bdot)$ is a species $\bmm$ together with a morphism of species $\mu: \bmm\bdot\bmm \to \bmm$, i.e., a family of maps 
\[
	\mu_{S,T} : \bmm[S] \otimes \bmm[T] \to \bmm[I],
\]
one for each ordered decomposition $I = S\djcup T$,  satisfying
appropriate associativity and unital conditions, and naturality with respect to bijections.
Briefly, to each $\bmm$-structure on $S$ and $\bmm$-structure on $T$, there is assigned an $\bmm$-structure on $S\djcup T$. 
The analogous object in the category of graded vector spaces is a graded algebra.

For the species $\be$, a monoid structure is defined by sending the basis element $\ast_S \otimes \ast_T$ to the basis element $\ast_I$. For $\bl$, a monoid structure is provided by concatenation of linear orders: $\mu_{S,T}(\ell_1 \otimes \ell_2) = (\ell_1, \ell_2)$. 

A \demph{comonoid} in the category $({\sf Sp},\bdot)$ is a species $\bc$ together with a morphism of species $\Delta:\bc \to \bc\bdot\bc$, i.e., a family of maps 
\[
	\Delta_{S,T} : \bc[I] \to \bc[S] \otimes \bc[T],
\]
one for each ordered decomposition $I=S\djcup T$, which are natural, coassociative and counital. 

For the species $\be$, a comonoid structure is defined by sending the basis vector $\ast_I$ to the basis vector $\ast_S \otimes \ast_T$. For $\bl$, a comonoid structure is provided by restricting a total order $\ell$ on $I$ to total orders on $S$ and $T$: $\Delta_{S,T}(\ell) = \ell\vert_S\otimes \ell\vert_T$. 

We assume that our species $\bq$ are \demph{connected}, i.e., $\bq[\emptyset]=\field$. 
In this case, the (co)unital conditions for a (co)monoid force the maps
$\mu_{S,T}$ ($\Delta_{S,T}$) to be the canonical identifications if either $S$ or $T$ is empty. Thus, in defining a connected (co)monoid structure one only needs to specify
the maps $\mu_{S,T}$ ($\Delta_{S,T}$) when both $S$ and $T$ are nonempty.

A \demph{Hopf monoid} in the category $({\sf Sp}, \bdot)$ is a monoid and comonoid whose two structures are compatible in an appropriate sense, and which carries an antipode. In this paper we only consider connected Hopf monoids. For such Hopf monoids, the existence of the antipode is guaranteed. The species $\be$ and $\bl$, with the structures outlined above, are two important examples of connected Hopf monoids.  

For further details on Hopf monoids in species, see Chapter 8 of \cite{AguMah:2010}.
The theory of Hopf monoids in species is developed in Part II of this reference;
several examples are discussed in Chapters 12 and 13.

\subsection{Lagrange's theorem for connected Hopf monoids} 

Given a connected Hopf monoid $\bk$ in species, we let $\bk_+$ denote the species
defined by
\[
\bk_+[I] = \begin{cases}
\bk[I] & \text{ if }I\neq\emptyset, \\
    0     & \text{ if }I=\emptyset.
\end{cases}
\]
If $\bk$ is a submonoid of a monoid $\bh$, then $\bk_+\bh$ denotes the right ideal
of $\bh$ generated by $\bk_+$. In other words,
\[
(\bk_+\bh)[I]=\sum_{\substack{S\sqcup T=I\\ S\neq\emptyset}}\mu_{S,T}\bigl(\bk[S]\otimes\bh[T]\bigr).
\]

\begin{theorem}\label{t:main} 
Let $\bh$ be a connected Hopf monoid in the category of species. If $\bk$
is a Hopf submonoid of
$\bh$, then $\bh$ is a free left $\bk$-module. Moreover,
\[
\bh\cong \bk\bdot(\bh/\bk_+\bh)
\]
as left $\bk$-modules (and as species).
\end{theorem} 

The proof is given after a series of preparatory results.
Our argument parallels Radford's proof of the first statement in Theorem~\ref{t:pointed} \cite[Section~4]{Rad:1977}.
The main ingredient is a result of Larson and
Sweedler \cite{LarSwe:1969} known as the fundamental theorem of Hopf modules 
\cite[Theorem~1.9.4]{Mon:1993}. It states that if $(M,\rho)$ is a left Hopf module over $K$, then $M$ is free as a left $K$-module and in fact is isomorphic to the Hopf module $K\otimes Q$, where $Q$ is the space of \demph{coinvariants} for the coaction $\rho \colon M\to K\otimes M$. 
Takeuchi extends this result to the context of Hopf monoids in a braided monoidal category with equalizers \cite[Theorem~3.4]{Tak:1999}; a similar result (in a more restrictive setting)
is given by Lyubashenko \cite[Theorem~1.1]{Lyu:1995}. As a special case
of Takeuchi's result, we have the following.

\begin{proposition}\label{p:FTHM} Let $\bmm$ be a left Hopf module over a connected Hopf monoid $\bk$ in species. There is an isomorphism $\bmm \cong \bk \bdot \bq$ of left Hopf modules, where 
\[
\bq[I] := \bigl\{m\in \bmm[I] \mid \text{$\rho_{S,T}(m) = 0$ for $S\sqcup T=I$, $T\neq I$} \bigr\}.
\]
In particular, $\bmm$ is free as a left $\bk$-module.
\end{proposition}
Here $\rho:\bmm\to\bk\bdot\bmm$ denotes the comodule structure, which consists of maps
\[
\rho_{S,T}: \bmm[I] \to \bk[S]\otimes\bmm[T],
\]
one for each ordered decomposition $I=S\sqcup T$.

\medskip

Given a comonoid $\bh$ and two subspecies $\bu,\bv \subseteq \bh$, the \demph{wedge} of $\bu$ and $\bv$ is the subspecies $\bu\wedge\bv$ of $\bh$ defined by
\[
\bu\wedge\bv := \Delta^{-1}(\bu\bdot\bh + \bh\bdot\bv).
\]

The remaining ingredients needed for the proof are supplied by the following lemmas.

\begin{lemma}\label{l:lemma 1}
Let $\bh$ be a comonoid in species.
If $\bu$ and $\bv$ are subcomonoids of $\bh$, then:
\begin{itemize}
\item[(i)] $\bu\wedge\bv$ is a subcomonoid of $\bh$ and $\bu+\bv\subseteq \bu\wedge\bv$;
\item[(ii)]
$\bu\wedge\bv= \Delta^{-1}\bigl(\bu\bdot(\bu\wedge\bv) + (\bu\wedge\bv)\bdot\bv\bigr)$.
\end{itemize}
\end{lemma}
\begin{proof}
(i) The proofs of the analogous statements for coalgebras given in \cite[Section~3.3]{Abe:1980}
extend to this setting. 

(ii) {}From the definition, $\Delta^{-1}\bigl(\bu\bdot(\bu\wedge\bv) + (\bu\wedge\bv)\bdot\bv\bigr)\subseteq \bu\wedge\bv$. Now, since $\bu\wedge\bv$ is a subcomonoid,
\[
\Delta(\bu\wedge\bv)\subseteq \bigr((\bu\wedge\bv)\bdot(\bu\wedge\bv)\bigr)\cap
(\bu\bdot\bh + \bh\bdot\bv) = \bu\bdot(\bu\wedge\bv) + (\bu\wedge\bv)\bdot\bv,
\]
since $\bu,\bv\subseteq \bu\wedge\bv$. This proves the converse inclusion.
\end{proof}

\begin{lemma}\label{l:lemma 1.5}
Let $\bh$ be a Hopf monoid in species and $\bk$ be a submonoid. 
Let $\bu,\bv\subseteq\bh$ be subspecies which are left $\bk$-submodules of $\bh$.
Then $\bu \wedge \bv$ is a left $\bk$-submodule of $\bh$.
\end{lemma}
\begin{proof}
Since $\bh$ is a Hopf monoid, the coproduct $\Delta:\bh\to\bh\bdot\bh$ is a morphism
of left $\bh$-modules, where $\bh$ acts on $\bh\bdot\bh$ via $\Delta$. Hence it is also 
a morphism of left $\bk$-modules. By hypothesis, $\bu\bdot\bh+\bh\bdot\bv$ is a left $\bk$-submodule
of $\bh\bdot\bh$. Hence, $\bu\wedge\bv=\Delta^{-1}(\bu\bdot\bh + \bh\bdot\bv)$ is a left $\bk$-submodule of $\bh$. 
\end{proof}

\begin{lemma}\label{l:lemma 2}
Let $\bh$ be a Hopf monoid in species and $\bk$ a Hopf submonoid. 
Let $\bc$ be a subcomonoid of $\bh$ and a left $\bk$-submodule of $\bh$.
Then $(\bk \wedge \bc)/\bc$ is a left $\bk$-Hopf module.
\end{lemma}

\begin{proof}
By Lemma~\ref{l:lemma 1.5}, $\bk\wedge\bc$ is a left $\bk$-submodule of $\bh$. Therefore, the quotient $(\bk\wedge\bc)/\bc$ by the left $\bk$-submodule $\bc$ is a left $\bk$-module.

We next argue that $(\bk\wedge\bc)/\bc$ is a $\bk$-comodule. Consider the composite
\[
\bk\wedge\bc \map{\Delta} \bk\bdot(\bk\wedge\bc) + (\bk\wedge\bc)\bdot\bc
\onto \bk\bdot\bigl(\bk\wedge\bc)/\bc,
\]
where the first map is granted by Lemma~\ref{l:lemma 1} and the second is the projection modulo $\bc$ on the second coordinate. Since $\bc$ is a subcomonoid,
the composite factors through $\bc$ and induces
\[
(\bk\wedge\bc)/\bc \to \bk\bdot\bigl(\bk\wedge\bc)/\bc.
\]
This defines a left $\bk$-comodule structure on $(\bk\wedge\bc)/\bc$.

Finally, the compatibility between the module and comodule structures on $(\bk\wedge\bc)/\bc$ is inherited from
the compatibility between the product and coproduct of $\bh$.
\end{proof}

We are nearly ready for the proof of the main result. First, recall the \demph{coradical filtration}  of a connected comonoid in species \cite[Section~8.10]{AguMah:2010}. Given a connected comonoid $\bc$, define subspecies $\bc_{(n)}$ by
\[
\bc_{(0)}=\tone
\qand
\bc_{(n)} = \bc_{(0)} \wedge \bc_{(n-1)} \text{ \ for all $n\geq 1$.}
\]
We then have 
\[
\bc_{(0)} \subseteq \bc_{(1)} \subseteq \dotsb\subseteq \bc_{(n)}\subseteq \dotsb \bc
\qand
\bc= \bigcup_{n\geq 0} \bc_{(n)}.
\]

\begin{proof}[Proof of Theorem~\ref{t:main}] We show that there is a species $\bq$ such that $\bh \cong \bk\bdot \bq$ as left $\bk$-modules. As in the proof of Theorem~\ref {t:connected}, it then follows that $\bq\cong \bh/\bk_+\bh$.

Define a sequence  $\bk^{(n)}$ of subspecies of $\bh$ by 
\[
\bk^{(0)}=\bk \qand \bk^{(n)} = \bk \wedge \bk^{(n-1)} \text{ \ for all $n\geq 1$.}
\] 

Each $\bk^{(n)}$ is a subcomonoid and a left $\bk$-submodule of $\bh$. 
This follows from Lemmas~\ref{l:lemma 1}
and~\ref{l:lemma 1.5} by induction on $n$.
Then Lemma~\ref{l:lemma 2} provides a left $\bk$-Hopf module structure
on the quotient species $\bk^{(n)}/\bk^{(n-1)}$ for all $n\geq1$. 
Hence $\bk^{(n)}/\bk^{(n-1)}$ is a free left $\bk$-module, by Proposition~\ref{p:FTHM}.

We claim that there exists a sequence of species $\bq_n$
such that
\[
\bk^{(n)} \cong \bk\bdot \bq_n
\]
as left $\bk$-modules for all $n\geq 0$; that is, each $\bk^{(n)}$ is a free left $\bk$-module. Moreover, we claim that the $\bq_n$ can be chosen so that
\[
\bq_0\subseteq \bq_1\ \subseteq \dotsb\subseteq \bq_n \subseteq\dotsb
\]
and the above isomorphisms are compatible with the inclusions $\bq_{n-1}\subseteq\bq_n$ and $\bk^{(n-1)}\subseteq\bk^{(n)}$.

We prove the claims by induction on $n\geq 0$. We start by letting $\bq_0=\tone$. 
For $n\geq 1$, we have
\[
\bk^{(n-1)}\cong \bk\bdot \bq_{n-1}
\qand
\bk^{(n)}/\bk^{(n-1)}\cong \bk\bdot \bq'_n
\] 
for some species $ \bq'_n$ (the former by induction hypothesis and the latter by the above argument). Let
\[
\bq_{n} = \bq_{n-1}\oplus \bq'_n.
\]
By choosing an arbitrary $\bk$-module section of the map
$\bk^{(n)}\onto\bk^{(n)}/\bk^{(n-1)}\cong \bk\bdot \bq'_n$ (possible by freeness),
we obtain an isomorphism 
\[
\bk^{(n)}\cong \bk\bdot \bq_{n}
\]
extending the isomorphism $\bk^{(n-1)}\cong \bk\bdot \bq_{n-1}$.
This proves the claims.

We utilize the coradical filtration of $\bh$ to finish the proof. Since $\bh$ is connected, $\bh_{(0)}=\tone \subseteq \bk= \bk^{(0)}$, and by induction,
\[
	\bh_{(n)} \subseteq \bk^{(n)} \quad\hbox{for all }n\geq0.
\]
Hence,
\[
\bh= \bigcup_{n\geq 0} \bh_{(n)}
=\bigcup_{n\geq 0} \bk^{(n)}
\cong\bigcup_{n\geq 0} \bk\bdot \bq_{n}
\cong\bk\bdot \bq
\text{ \ where \ }
\bq=\bigcup_{n\geq 0} \bq_n.
\] 
Thus $\bh$ is free as a left $\bk$-module. 
\end{proof}

Let $\pi:\bh\to\bk$ be a morphism of Hopf monoids. The \demph{right Hopf kernel} of $\pi$ is the species defined by
\begin{equation}\label{e:hker}
\hker(\pi)=\ker\bigl(\bh \map{\Delta} \bh\bdot\bh\map{\pi_{+}\bdot\,\id} \bk_{+\!}\bdot\bh\bigr),
\end{equation}
where $\pi_+:\bh\to\bk_+$ is $\pi$ followed by the canonical projection $\bk\onto\bk_+$.

For the following result, we employ duality for Hopf monoids~\cite[Section~8.6.2]{AguMah:2010}. (We assume all species are finite dimensional.)

\begin{theorem}\label{t:maindual} 
Let $\bh$ be a connected Hopf monoid in the category of species and $\bk$
a quotient Hopf monoid via a morphism $\pi:\bh\onto\bk$.
Then $\bh$ is a cofree left $\bk$-comodule. Moreover,
\[
\bh\cong \bk\bdot \hker(\pi)
\]
as left $\bk$-comodules (and as species).
\end{theorem} 
\begin{proof} By duality,
$\bk^*$ is a Hopf submonoid of $\bh^*$, so $\bh^*\cong\bk^*\bdot(\bh^*/\bk^*_+\bh^*)$
by Theorem~\ref{t:main}.
Dualizing again we obtain the result, since
\[
\bh^*/\bk^*_+\bh^*=\coker\bigl(\bk^*_{+\!}\bdot\bh^* \map{\pi^*_+\bdot\,\id} 
\bh^{*\!}\bdot\bh^* \map{\Delta^*} 
\bh^* \bigr). \qedhere
\]
\end{proof}

\section{Applications and examples}\label{s:applications}

\subsection{A test for Hopf submonoids}

Two invariants associated to a (finite dimensional) species $\bq$ are the \demph{exponential generating series} $\exponential{\bq}{x}$ and the
\demph{type generating series} $\type{\bq}{x}$. They
are given by
\[
\exponential{\bq}{x} = \sum_{n\geq 0} \dim \bq[n]\, \frac{x^n}{n!}
\qand
\type{\bq}{x} = \sum_{n\geq 0} \dim \bq[n]_{S_n}\, x^n,
\]
where 
\[
\bq[n]_{S_n} =\bq[n]/\field\{v-\sigma\cdot v \mid v\in \bq[n],\ \sigma\in S_n\}. 
\]
Both are specializations of the \demph{cycle index series} $\cycle{\bq}{x_1,x_2,\dotsc}$; see \cite[Section~1.2]{BerLabLer:1998} for the definition. Specifically, 
\begin{gather*}
	\exponential{\bq}{x} = \cycle{\bq}{x,0,0,\dotsc}
\qand
	\type{\bq}{x} = \cycle{\bq}{x,x^2,x^3,\dotsc} .
\end{gather*}

The cycle index series is multiplicative under Cauchy product \cite[Section~1.3]{BerLabLer:1998}: if $\bh = \bk \bdot \bq$, then $\cycle{\bh}{x_1,x_2,\dotsc} = \cycle{\bk}{x_1,x_2,\dotsc} \, \cycle{\bq}{x_1,x_2,\dotsc}$. By specialization, the same is true for the exponential
and type generating series.

Let $\bQ_{\geq 0}$ denote the nonnegative rational numbers. An immediate consequence of Theorems~\ref{t:main} and~\ref{t:maindual} is the following.

\begin{corollary}\label{c:Sp series} 
Let $\bh$ and $\bk$ be connected Hopf monoids in species. 
Suppose $\bk$ is either a Hopf submonoid or a quotient Hopf monoid of $\bh$.
Then the quotient of cycle index series
$\cycle{\bh}{x_1,x_2,\dotsc} / \cycle{\bk}{x_1,x_2,\dotsc}$ is nonnegative, i.e., belongs to $\bQ_{\geq 0}\llb x_1,x_2,\dotsc \rrb$. 
In particular, the quotients $\exponential{\bh}{x}/\exponential{\bk}{x}$ and $\type{\bh}{x}/\type{\bk}{x}$ are also nonnegative.
\end{corollary}

Given a connected Hopf monoid $\bh$ in species, we may use Corollary~\ref{c:Sp series} to determine if a given species $\bk$ may be a Hopf submonoid (or a quotient Hopf monoid).

\begin{example}\label{ex: submonoid}
A \demph{partition} of a set $I$ is an unordered
collection of disjoint nonempty subsets of $I$ whose union is $I$. The notation 
$ab\pmrg c$ is shorthand for the partition $\bigl\{\{a,b\},\,\{c\}\bigr\}$ of $\{a,b,c\}$.

Let $\bm\Pi$ be the species of set partitions, i.e., $\bm\Pi[I]$ is the vector space with basis the set of all partitions of $I$.
Let $\bm\Pi'$ denote the subspecies linearly spanned by set partitions with distinct block sizes. For example, 
\[
\bm\Pi[a,b,c] = \field\bigl\{ abc, a\pmrg bc, ab\pmrg c, a\pmrg bc, a\pmrg b\pmrg c \bigr\} 
\qand
\bm\Pi'[a,b,c] = \field\bigl\{ abc, a\pmrg bc, ab\pmrg c, a\pmrg bc\bigr\}.
\]
The sequences $(\dim\bm\Pi[n])_{n\geq 0}$ and $(\dim\bm\Pi'[n])_{n\geq 0}$ appear in \cite{Slo:oeis} as  A000110 and A007837, respectively. We have
\[
\exponential{\bm\Pi}{x} = \exp\bigl(\exp(x)-1\bigr) = 1+x+x^2+\frac{5}{6}x^3+\frac{5}{8}x^4+\dotsb
\]
and
\[
\exponential{\bm\Pi'}{x} = \prod_{n\geq 1}\bigl(1+\frac{x^n}{n!}\bigr) = 1+x+\frac{1}{2}x^2+\frac{2}{3}x^3+\frac{5}{24}x^4+\dotsb \,.
\]

A Hopf monoid structure on $\bm\Pi$ is defined in \cite[Section~12.6]{AguMah:2010}. 
There are many linear bases of $\bm\Pi$ indexed by set partitions, and many ways
to embed $\bm\Pi'$ as a subspecies of $\bm\Pi$. Is it possible to embed $\bm\Pi'$ as a 
Hopf submonoid of $\bm\Pi$? We have
\[
\exponential{\bm\Pi}{x} \big/ \exponential{\bm\Pi'}{x} 
 =  
1+\frac12 x^2 - \frac1{3} x^3 + \frac{1}{2} x^4 - \frac{11}{30} x^5 + \dotsb \,,
\]
so the answer is negative by Corollary~\ref{c:Sp series}. In fact, it is not possible to embed $\bm\Pi'$ as a Hopf submonoid of $\bm\Pi$ for any potential Hopf monoid structure on $\bm\Pi$.

We remark that the type generating series quotient for the pair of species in Example~\ref{ex: submonoid} is nonnegative:
\begin{align*}
	\type{{\bm\Pi}}{x}  \ &= \ 1+x+2x^2+3x^3+5x^4+7x^5+11x^6+15x^7  + \dotsb \,, \\
	\type{{\bm\Pi'}}{x} \ &= \ 1 + x+x^2+2x^3+2x^4+3x^5+4x^6+5x^7  + \dotsb \,, \\
	\type{{\bm\Pi}}{x} \big/ \type{{\bm\Pi'}}{x} \ &= \ 1+x^2+2x^4+3x^6+5x^8+7x^{10} + \dotsb \,. 
\end{align*}
This can be understood by appealing to the Hopf algebra $\sym$ of symmetric functions. A basis for its homogenous component of degree $n$ is indexed by integer partitions of $n$, so $\poincare{\sym}{x} = \type{{\bm\Pi}}{x}$. Moreover, $\type{{\bm\Pi'}}{x}$ enumerates the integer partitions with odd part sizes and $\sym$ does indeed contain a Hopf subalgebra with this Poincar\'e series. It is the algebra of Schur $Q$-functions. See \cite[III.8]{Mac:1995}. Thus $\type{{\bm\Pi}}{x} \big/ \type{{\bm\Pi'}}{x}$ is nonnegative by Corollary~\ref{c:cgVec series}.
\end{example}

\subsection{Tests for Hopf monoids}\label{ss:ord-exp}

Let $(a_n)_{n\geq 0}$ be a sequence of nonnegative integers with $a_0=1$.
Does there exist a connected Hopf monoid $\bh$ with $\dim \bh[n]=a_n$ for all $n$?
The next result provides conditions that the sequence $(a_n)_{n\geq 0}$ must
satisfy in order for this to be the case.
The proof makes use of the \demph{Hadamard product\/} of Hopf monoids~\cite[Sections~8.1 and~8.13]{AguMah:2010}. If $\bh$ and $\bk$ are Hopf monoids, so is $\bh\times\bk$, with $(\bh\times\bk)[I] = \bh[I]\otimes \bk[I]$ for each finite set $I$.
The exponential species $\be$ is the unit element for the Hadamard product. 

\begin{corollary}[The (ord/exp)-test]\label{c:ord-exp}
For any connected Hopf monoid in species $\bh$, 
\[
\Bigl(\sum_{n\geq 0} \dim \bh[n]\, x^n\Bigr) \Big/ \Bigl(\sum_{n\geq 0} \frac{\dim \bh[n]}{n!} \, x^n\Bigr)   \in \bQ_{\geq 0}\llb x\rrb.
\]
\end{corollary}
\begin{proof}
Consider the canonical morphism of Hopf monoids $\bl \onto \be$~\cite[Section~8.5]{AguMah:2010}; it maps any linear order $\ell\in\bl[I]$ to the basis element $\ast_I\in\be[I]$. The Hadamard product then yields a morphism of
Hopf monoids
\[
\bl\times\bh \onto \be\times\bh\cong\bh.
\]
By Corollary~\ref {c:Sp series}, $\exponential{\bl\times\bh}{x}/\exponential{\bh}{x}\in
\bQ_{\geq 0}\llb x\rrb$. Since $\exponential{\bl\times\bh}{x}=\sum_{n\geq 0} \dim \bh[n]\, x^n$, the result follows.
\end{proof}

Let $a_n=\dim \bh[n]$. Corollary~\ref{c:ord-exp} states that the ratio of the ordinary to the exponential generating function of the sequence $(a_n)_{n\geq 0}$ must be nonnegative.
This translates into a sequence of polynomial inequalities, the first of which are as follows:
\begin{equation}\label{e:ord-exp}
5a_3\geq 3a_2a_1, \quad
23a_4+12a_2a_1^2 \geq 20a_3a_1+6a_2^2.
\end{equation}
In particular, not every nonnegative sequence arises as the dimension sequence of a Hopf monoid.

The following test is of a similar nature, but involves the type 
instead of the exponential generating function. The conditions then
depend not just on the dimension sequence of $\bh$, but also on its species structure.

\begin{corollary}[The (ord/type)-test]\label{c:full-bosonic}
For any connected Hopf monoid in species $\bh$, 
\[
\Bigl(\sum_{n\geq 0} \dim \bh[n]\, x^n\Bigr) \Big/ \Bigl(\sum_{n\geq 0} \dim \bh[n]_{S_n}\, x^n\Bigr)   \in \bN\llb x\rrb.
\] 
\end{corollary}
\begin{proof}
We argue as in the proof of Corollary~\ref{c:ord-exp}, using type generating functions
instead. Since we have $\type{\bl\times\bh}{x}= \sum_{n\geq 0} \dim \bh[n]\, x^n$, the result follows.
\end{proof}

\begin{remark}
The previous result may also be derived as follows. According to~\cite[Chapter 15]{AguMah:2010},
associated to the Hopf monoid $\bh$ there are two graded Hopf algebras $\K(\bh)$ and $\Kb(\bh)$, as well as
a surjective morphism
\[
\K(\bh) \onto \Kb(\bh).
\]
Moreover, the Poincar\'e series for these Hopf algebras are
\[
\poincare{\K(\bh)}{x} = \sum_{n\geq 0} \dim \bh[n]\, x^n
\qand
\poincare{\Kb(\bh)}{x} = \sum_{n\geq 0} \dim \bh[n]_{S_n}\, x^n.
\]
Corollary~\ref{c:full-bosonic} now follows from (the dual form of) Corollary~\ref{c:cgVec series}.
\end{remark}

\subsection{Additional tests for Hopf monoids}\label{ss:addtests}

The method of Section~\ref{ss:ord-exp} can be applied in multiple situations in order 
to deduce
additional inequalities that the dimension sequence of a connected Hopf monoid must satisfy. We illustrate this next.

Let $k$ be a fixed nonnegative integer. Let $\be^{\bdot k}$ denote the $k$-th Cauchy power of the exponential species $\be$. The space $\be^{\bdot k}[I]$ has a basis
consisting of functions $f:I\to [k]$. The species $\be^{\bdot k}$
carries a Hopf monoid structure~\cite[Examples~8.17 and~8.18]{AguMah:2010}
and any fixed inclusion $[k]\hookrightarrow[k{+}1]$ gives rise to an injective morphism
of Hopf monoids $\be^{\bdot k}\hookrightarrow\be^{\bdot (k+1)}$. Employing the
Hadamard product as in Section~\ref{ss:ord-exp}, we obtain an injective morphism
of Hopf monoids
\[
\be^{\bdot k}\times\bh \hookrightarrow \be^{\bdot (k+1)}\times\bh
\]
where $\bh$ is an arbitrary connected Hopf monoid. From the nonnegativity of
the first coefficients of $\exponential{\be^{\bdot (k+1)}\times\bh}{x} \big/ \exponential{\be^{\bdot k}\times\bh}{x}$ we obtain
\[
(2k+1)a_2 \geq 2k a_1^2
\qand
(3k^2+3k+1)a_3 \geq 3(3k^2+k)a_2a_1 - 6k^2a_1^3.
\]
These inequalities hold for every $k\in\bN$. Letting $k\to\infty$ we deduce
\begin{equation}\label{e:addcond}
a_2\geq a_1^2
\qand
a_3\geq 3a_2a_1-2a_1^3.
\end{equation}

\begin{example}
Consider the species $\bel$ \demph{of elements}. The set $I$ is a basis of the space $\bel[I]$, so the dimension sequence of $\bel$ is $a_n=n$. This sequence does not
satisfy the second inequality in~\eqref{e:addcond}. Therefore, the species $\bel$ does not carry any Hopf
monoid structure.
\end{example}

\subsection{A test for Hopf monoids over $\be$}
Our next result is a necessary condition for a Hopf monoid in species to contain or surject onto the exponential species $\be$. 

Given a sequence $(a_n)_{n\geq0}$, its \demph{binomial transform} $(b_n)_{n\geq0}$ is defined by
\[
b_n := \sum_{i=0}^n \binom{n}{i}(-1)^{i}\,a_{n-i}.
\]

\begin{corollary}[The $\be$-test] \label{c:e-test}
Suppose $\bh$ is a connected Hopf monoid that either contains the species $\be$ or surjects onto $\be$ (in both cases as a Hopf monoid). Let $a_n=\dim \bh[n]$ and $\overline{a}_n=\dim \bh[n]_{S_n}$. Then the binomial transform of $(a_n)_{n\geq0}$ 
must be nonnegative and $(\overline{a}_n)_{n\geq0}$ must be nondecreasing.
\end{corollary}

More plainly, in this setting, we must have the following inequalities:
\[
a_1\geq a_0, \quad 
a_2\geq 2a_1-a_0, \quad
a_3\geq 3a_2-3a_1+a_0, \ \ \dotsc
\]
and $\overline{a}_n\geq \overline{a}_{n-1}$ for all $n\geq 1$.

\begin{proof}
By Corollary~\ref{c:Sp series}, the quotient $\exponential{\bh}{x} / \exponential{\be}{x}$ is nonnegative. But $\exponential{\be}{x} = \exp(x)$,
so the quotient is given by
\[
	b_0 + b_1 x + b_2 \frac{x^2}{2} + b_3 \frac{x^3}{3!} + \dotsb \,,
\]
where $(b_n)_{n\geq0}$ is the binomial transform of $(a_n)_{n\geq0}$. Replacing
exponential for type generating functions yields the
result for $(\overline{a}_n)_{n\geq0}$, since $\type{\be}{x} = \frac{1}{1-x}$.
\end{proof}

\begin{remark}
Myhill's theory of \demph{combinatorial functions}~\cite{Dek:1990,Myh:1958} 
provides necessary and sufficient conditions that a sequence $(a_n)_{n\geq0}$ must satisfy in order for its binomial transform to be nonnegative:
the sequence must arise from a particular type of operator defined on finite sets.
Work of Menni~\cite{Men:2009} expands on this from a categorical perspective. It would be interesting to
relate these ideas to the ones of this paper.
\end{remark}

We make a further remark regarding connected \demph{linearized} Hopf monoids. 
These are Hopf monoids of a set theoretic nature. See \cite[Section~8.7]{AguMah:2010} for details. Briefly, there are set maps $\mu_{S,T}:\rH[S]\times\rH[T]\to\rH[I]$ and $\Delta_{S,T}:\rH[I]\to\rH[S]\times\rH[T]$ which produce a Hopf monoid in (vector) species when the set species $\rH$ is linearized. It follows that if $\bh$ is a linearized Hopf monoid other than the trivial Hopf monoid $\tone$, then there is a morphism of Hopf monoids from $\bh$ onto $\be$. 
Thus, Corollary~\ref{c:e-test} provides a test for existence of a linearized Hopf monoid structure on $\bh$.

\begin{example} 
We return to the species $\bm\Pi'$ of set partitions into distinct block sizes. We might ask if this can be made into a linearized Hopf monoid in some way (after Example~\ref{ex: submonoid}, this would \emph{not} be as a Hopf submonoid of $\bm\Pi$). With $a_n$ and $b_n$ as above, we have:
\begin{align*}
(a_n)_{n\geq0} \ &= \ 1, \ 1, \ 1,  \ 4,  \ \ 5,  \ 16,  \ \ 82, \ \ 169, \ 541, \dotsc \,, \\[.5ex]
(b_n)_{n\geq0} \ &= \ 1, \ 0, \ 0, \ 3, \, -8,\  25, \, -9, \, -119, \ 736, \dotsc \,.
\end{align*}
Thus $\bm\Pi'$ fails the $\be$-test and the answer to the above question is negative.
\end{example}

\subsection{A test for Hopf monoids over $\bl$}
Let $\bh$ be a connected Hopf monoid in species. Let $a_n=\dim \bh[n]$ and $\overline{a}_n=\dim \bh[n]_{S_n}$. Note that the analogous integers for the species $\bl$ of linear orders are $b_n=n!$ and $\overline{b}_n=1$. Now suppose that $\bh$ contains $\bl$ or surjects onto $\bl$ as a Hopf monoid.
An obvious necessary condition for this situation is that $a_n\geq n!$ and $\overline{a}_n\geq 1$. Our next result provides stronger conditions.

\begin{corollary}[The $\bl$-test] \label{c:l-test}
Suppose $\bh$ is a connected Hopf monoid that either contains the species $\bl$ or surjects onto $\bl$ (in both cases as a Hopf monoid). If $a_n=\dim \bh[n]$ and $\overline{a}_n=\dim \bh[n]_{S_n}$, then 
\[
a_n\geq n a_{n-1} \qand \overline{a}_n\geq \overline{a}_{n-1} \quad(\forall\,n\geq1).
\] 
\end{corollary}

\begin{proof}
It follows from Corollary~\ref{c:Sp series} that both $\exponential{\bh}{x} / \exponential{\bl}{x}$ and $\type{\bh}{x} / \type{\bl}{x}$ are nonnegative. These yield the first 
and second 
set of inequalities, respectively.
\end{proof}

Before giving an application of the corollary, we introduce a new Hopf monoid in species. A \demph{composition} of a set $I$ is an ordered
collection of disjoint nonempty subsets of $I$ whose union is $I$. The notation 
$ab\cmrg c$ is shorthand for the composition $\bigl(\{a,b\},\,\{c\}\bigr)$ of $\{a,b,c\}$.

Let $\mathbf{Pal}$ denote the species of set compositions whose sequence of block sizes is palindromic. So, for instance,
\begin{gather*}
\mathbf{Pal}[a,b] = \field\bigl\{ ab, \ a\cmrg b, \ b\cmrg a\bigr\}\\
\intertext{and}
\mathbf{Pal}[a,b,c,d,e] = \field\bigl\{abcde, \ a\cmrg bcd\cmrg e, \ ab\cmrg c\cmrg de, \ a\cmrg b\cmrg c\cmrg d\cmrg e, \ \dotsc \bigr\}.
\end{gather*}
The latter space has dimension $171=1 + 5\binom{4}{3} + \binom{5}{2}3 + 5!$
while $\dim \mathbf{Pal}[5]_{S_5}=4$ (accounting for the four types of palindromic set compositions shown above).

Given a palindromic set composition $F=F_1\cmrg \dotsb\cmrg F_r$, we view it as a triple $F=(F^-,F^0,F^+)$, where $F^-$ is the initial sequence of blocks, $F^0$ is the central block if this exists (if the number of blocks is odd) and otherwise it is the empty set, and $F^+$ is the final sequence of blocks. That is,
\[
	F^-=F_1\cmrg \dotsb \cmrg F_{\lfloor r/2\rfloor}, 
	\qquad\quad 
	F^0 = \begin{cases} F_{\lfloor r/2\rfloor+1} & \hbox{if $r$ is odd,}\\
	\emptyset & \hbox{if $r$ is even,} \end{cases}
	\qquad\quad
	F^+ = F_{\lceil r/2+1\rceil}\cmrg \dotsb \cmrg F_{r}.
\]

Given $S\subseteq I$, let  
\[
F\vert_S := F_1\cap S\,\cmrg\, F_2\cap S\,\cmrg\, \dotsb \,\cmrg\, F_r\cap S\,,
\]
where empty intersections are deleted. 
Then $F\vert_S$ is a composition of $S$. 
Let us say that $S$ is \demph{admissible} for $F$ if
\[
	\#\bigl( F_i\cap S\bigr) = \# \bigl(F_{r+1-i}\cap S\bigr) \ \text{ \ for all $i=1,\ldots,r$.}
\]
In this case, both $F\vert_S$ and $F\vert_{I\setminus S}$ are palindromic. 

We employ the above notation to define product and coproduct operations on $\mathbf{Pal}$. Fix a decomposition $I = S\djcup T$. \\
{\it Product.\ } 
Given palindromic set compositions $F\in\mathbf{Pal}[S]$ and $G\in\mathbf{Pal}[T]$, we put
\[
\mu_{S,T}(F \otimes G) := \bigl(F^-\cmrg G^-, F^0\cup G^0, G^+\cmrg F^+\bigr).
\]
In other words, we concatenate the initial sequences of blocks of $F$ and $G$ in that order,
merge their central blocks, and concatenate their final sequences in the opposite order.
The result is a palindromic composition of $I$. For example, with $S=\{a,b\}$ and 
$T=\{c,d,e,f\}$,
\[
(a\cmrg b) \otimes (c\cmrg de\cmrg f) \mapsto a\cmrg c\cmrg de\cmrg f\cmrg b.
\]
{\it Coproduct.\ } Given a palindromic set composition $F\in\mathbf{Pal}[I]$, we put
\[
\Delta_{S,T}(F) := \begin{cases}
	F\vert_S \otimes F\vert_T & \hbox{if $S$ is admissible for $F$,} \\
	0 & \hbox{otherwise.}
	\end{cases} 
\]
For example, with $S$ and $T$ as above,
\[
ad\cmrg b\cmrg e\cmrg cf \mapsto 0
\qand
e\cmrg abcd\cmrg f \mapsto (ab) \otimes (e\cmrg cd\cmrg f).
\]
These operations endow $\mathbf{Pal}$ with the structure of Hopf monoid, as may
be easily checked.

\begin{example} 
A linear order may be seen as a palindromic set composition (with singleton blocks).
Both Hopf monoids $\mathbf{Pal}$ and $\bl$ are cocommutative and not commutative. 
We may then ask if $\mathbf{Pal}$ contains (or surjects onto) $\bl$ as a Hopf monoid. 

Writing $a_n=\dim\mathbf{Pal}[n]$, we have:
\begin{gather*}
(a_n)_{n\geq0} \ = \ 1, \ 1, \ 3, \ 7, \ 43, \ 171, \ 1581,  \ 8793, \ 108347, \dotsc \,.
\end{gather*}
However,
\begin{gather*}
(a_n - na_{n-1})_{n\geq1} \ = \ 0, \ 1, \  -2, \  15, \ -44, \  555, \ -2274, \ \ 38003, \dotsc \,,
\end{gather*}
so $\mathbf{Pal}$ fails the $\bl$-test and the answer to the above question is negative.
\end{example}

\subsection{Examples of nonnegative quotients}
We comment on a few examples where the quotient
power series $\exponential{\bh}{x}/\exponential{\bk}{x}$ is not only nonnegative
but is known to have a combinatorial interpretation as a generating function.

\begin{example}
Consider the Hopf monoid $\bm\Pi$ of set partitions. 
It contains $\be$ as a Hopf submonoid via the map that sends $\ast_I$ to the partition of $I$ into singletons.
We have
\[
\exponential{\bm\Pi}{x}/\exponential{\be}{x} = \exp\bigl(\exp(x)-x-1\bigr),
\]
which is the exponential generating function for the number of set partitions into blocks of size strictly bigger than $1$. 
This fact may also be understood with the aid of Theorem~\ref{t:main}, as follows. 

Given $I=S\sqcup T$, the product of a partition $\pi\in\bm\Pi[S]$
and a partition $\rho\in\bm\Pi[T]$ is the partition $\pi\cdot\rho\in\bm\Pi[I]$ each of whose blocks
is either a block of $\pi$ or a block of $\rho$. (In the notation of \cite[Section~12.6]{AguMah:2010},
we are employing the $h$-basis of $\bm\Pi$.)
Now, the $I$-component of the right ideal $\be_+\bm\Pi$ is linearly spanned by elements of the
form $\ast_S\cdot\pi$ where $I=S\sqcup T$ and $\pi$ is a partition of $T$. Then, since
$\ast_S=\ast_{\{i\}}\cdot\ast_{S\setminus\{i\}}$ (for any $i\in S$), we have that
$\be_+\bm\Pi[I]$ is linearly spanned by elements of the form $\ast_{\{i\}}\cdot\pi$ where
$i\in I$ and $\pi$ is a partition of $I\setminus\{i\}$. But the above description of the product shows that these are precisely the partitions
with at least one singleton block. 
\end{example}

\begin{example}\label{eg:derangement} 
Consider the Hopf monoids $\bl$ and $\be$ and the surjective
morphism $\pi:\bl\onto\be$ (as in the proof of Corollary~\ref{c:ord-exp}).
We have
\[
\exponential{\bl}{x}=\frac{1}{1-x}
\qand
\exponential{\be}{x}=\exp(x).
\]
It is well-known~\cite[Example~1.3.9]{BerLabLer:1998} that
\[
\frac{\exp(-x)}{1-x} = \sum_{n\geq 0} \frac{d_n}{n!}\,x^n
\]
where $d_n$ is the number of \demph{derangements} of $[n]$. Together with Theorem~\ref{t:maindual},
this suggests the existence of a basis for the Hopf kernel of $\pi$ indexed by derangements.
We construct such a basis and expand on this discussion in Section~\ref{ss:derangement}.
\end{example}

\begin{example}
Let $\bm{\Sigma}$ be the Hopf monoid of set compositions defined in~\cite[Section~12.4]{AguMah:2010}. It contains $\bl$ as a Hopf submonoid via the map that views a linear order as a composition into singletons. (In the notation of \cite[Section~12.4]{AguMah:2010},
we are employing the $H$-basis of $\bm{\Sigma}$.)
This and other morphisms relating $\be$, $\bl$, 
$\bm{\Pi}$ and $\bm{\Sigma}$, as well as other Hopf monoids, are discussed in~\cite[Section~12.8]{AguMah:2010}.

The sequence $(\dim \bm{\Sigma}[n])_{n\geq 0}$ is A000670 in \cite{Slo:oeis}. We have
\[
\exponential{\bm{\Sigma}}{x}=\frac{1}{2-\exp(x)}.
\]
Moreover, it is known from~\cite[Exercise~5.4.(a)]{Sta:1999} that
\[
\frac{1-x}{2-\exp(x)} = \sum_{n\geq 0} \frac{s_n}{n!}\,x^n
\]
where $s_n$ is the number of \demph{threshold} graphs with vertex set $[n]$ 
and no isolated vertices. Together with Theorem~\ref{t:main},
this suggests the existence of a basis for  $ \bm{\Sigma}/\bl_+ \bm{\Sigma}$ indexed by such graphs.
We do not pursue this possibility in this paper.
\end{example}

\section{The dimension sequence of a connected Hopf monoid}\label{s:dimensions}

Let $\bh$ be a connected Hopf monoid and $a_n=\dim\bh[n]$ for $n\in\bN$.
The results of Sections~\ref{ss:ord-exp} and~\ref{ss:addtests}, derived from 
Theorem~\ref{t:main},
 impose restrictions on the sequence $a_n$ in the form of polynomial inequalities.
The results of this section are neither weaker nor stronger than those of Section~\ref{s:applications},
but provide supplementary information on the dimension sequence $a_n$.
They do not make use of Theorem~\ref{t:main}. In this section, the base field $\field$ is of arbitrary characteristic.

\begin{proposition}\label{p:aiaj}
For any $n,i$ and $j$ such that $n=i+j$,
\begin{equation}\label{e:aiaj}
a_n\geq a_ia_j.
\end{equation}
\end{proposition}
\begin{proof} 
Since $\bh$ is connected, the compatibility axiom for Hopf monoids (diagram (8.18) 
in~\cite[Section~8.3.1]{AguMah:2010}) implies that the composite
\[
\bh[S]\otimes\bh[T] \map{\mu_{S,T}} \bh[I] \map{\Delta_{S,T}} \bh[S]\otimes\bh[T]
\]
is the identity. The result follows by choosing any decomposition $I=S\sqcup T$ with
$\abs{I}=n$, $\abs{S}=i$, and $\abs{T}=j$.
\end{proof}

\begin{remark}
The second inequality in~\eqref{e:addcond} may be combined with~\eqref{e:aiaj}
to obtain
\[
a_3-a_2a_1 \geq 2a_1(a_2-a_1^2)\geq 0.
\]
Considerations of this type show that neither set of inequalities~\eqref{e:ord-exp},~\eqref{e:addcond} or~\eqref{e:aiaj} follows from the others.
\end{remark}

As a first consequence of Proposition~\ref{p:aiaj}, we derive a result on the growth
of the dimension sequence.

\begin{corollary}\label{c:expgrowth}
If $a_1\geq 1$, then the sequence $a_n$ is weakly increasing. If moreover there exists $k\geq 1$ such that 
\[
a_k\geq 2 \qand a_i\geq 1\ \forall\, i=0,\ldots,k-1,
\]
then $a_n=O(2^{n/k})$.
\end{corollary}
\begin{proof}
The first statement follows from $a_n\geq a_1a_{n-1}$. Now fix $k$ as in the second statement.
Given $n\geq k$, write $n=qk+r$ with $q\in\bN$ and $0\leq r\leq k-1$. 
From~\eqref{e:aiaj} we obtain
\[
a_n\geq a_k^q a_r\geq 2^q = 2^{-r/k}2^{n/k}.
\]
Thus $a_n=O(2^{n/k})$. 
\end{proof}

Define the \demph{support} of $\bh$ to be the support of its dimension sequence; namely, 
\[
\supp(\bh)=\{n\in\bN \mid a_n\neq 0\}.
\]
We turn to consequences of Proposition~\ref{p:aiaj} on the support.

\begin{corollary}\label{c:aiaj}
The set $\supp(\bh)$ is a submonoid of $(\bN,+)$.
\end{corollary}
\begin{proof}
By (co)unitality of $\bh$, $0\in \supp(\bh)$. (In fact, $a_0=1$ by connectedness.) By Proposition~\ref{p:aiaj}, the set $\supp(\bh)$ is
closed under addition. 
\end{proof}

We mention that, conversely, given any submonoid $S$ of $(\bN,+)$, there exists
a connected Hopf monoid $\bh$ such that $\supp(\bh)=S$.
Indeed, let $\bm\Pi_S[I]$ be the space spanned by the set of partitions of $I$ whose
block sizes belong to $S\setminus\{0\}$. Then $\bm\Pi_S$ is a quotient Hopf monoid of $\bm\Pi$ and $\supp(\bm\Pi_S)=S$.
(The former follows from the formulas in~\cite[Section~12.6.2]{AguMah:2010}; we employ the $h$-basis of $\bm\Pi$.)

\begin{example} Consider the special case of the previous paragraph in which $S$ is the
submonoid of even numbers. Then $\Pi_S$ is the species of set partitions into blocks of even size. In particular, $a_n=0$ for all odd $n$, so the dimension sequence
is neither increasing nor of exponential growth. This example shows that the
hypotheses of Corollary~\ref{c:expgrowth} cannot be removed.
\end{example}

\begin{corollary}\label{c:numerical}
The set $\supp(\bh)$ is either $\{0\}$ or infinite. 
The set $\bN\setminus\supp(\bh)$ is finite if and only if $\gcd\bigl(\supp(\bh)\bigr)=1$.
\end{corollary}
\begin{proof}
These statements hold for all submonoids of $\bN$, hence for $\supp(\bh)$ by Corollary~\ref{c:aiaj}. For the second statement, see~\cite[Lemma~2.1]{RosGar:2009}.
\end{proof}

\begin{remark}
We comment on counterparts for  connected graded Hopf algebras of the results of this section. 

Consider the polynomial Hopf algebra $H=\field[x_1,\ldots,x_k]$, in which the generators $x_i$ are primitive and of degree $1$. The dimension sequence is $a_n=\binom{n+k-1}{k-1}$. In contrast to Corollary~\ref{c:expgrowth},
this sequence is polynomial even if $a_1>1$. It follows that Proposition~\ref{p:aiaj}
has no counterpart for connected graded Hopf algebras $H$, and that the multiplication map $H_i\otimes H_j \to H_{i+j}$ is not injective in general.

Corollaries~\ref{c:aiaj} and~\ref{c:numerical} fail for connected Hopf algebras over a field of positive characteristic.
In characteristic $p$, a counterexample is provided by $H=\field[x]/(x^p)$ with $x$ primitive and of degree $1$.

On the other hand, if the field characteristic is $0$, then the set $S=\{n\in\bN \mid H_n\neq 0\}$ is a submonoid of $(\bN,+)$.
This follows from the fact that in this case any connected Hopf algebra is a domain.
We expand on this point in the Appendix.


 \end{remark}

\section{Hopf kernels for cocommutative Hopf monoids}\label{s:kernels}

Hopf kernels enter in the decomposition of Theorem~\ref{t:maindual}
(and in dual form, in Theorem~\ref{t:main}). For cocommutative Hopf monoids,
Hopf kernels and Lie kernels are closely related, as discussed in this section. 
We provide a simple result that allows us to describe the Hopf kernel in certain
situations and we illustrate it with the case of the canonical morphism $\bl\onto\be$.

\subsection{Hopf and Lie kernels}
The species $\Pc(\bh)$ of \demph{primitive elements} of a connected Hopf monoid $\bh$ 
is defined by
\[
\Pc(\bh)[I] =
\{ x \in \bh[I] \mid \Delta(x) = 1 \otimes x + x \otimes 1\}
\]
for each nonempty finite set $I$, and $\Pc(\bh)[\emptyset]=0$.
Equivalently,
\[
\Pc(\bh)[I] = \bigcap_{\substack{S\sqcup T=I \\ S,T\neq\emptyset}} \ker\bigl(\Delta_{S,T}:\bh[I]\to \bh[S]\otimes\bh[T]\bigr).
\]
It is a Lie submonoid of $\bh$ under the commutator bracket.
See~\cite[Sections~8.10 and~11.9]{AguMah:2010} for more information on primitive elements. 

Let $\pi:\bh\to\bk$ be a morphism of connected Hopf monoids. It restricts to a morphism of Lie monoids
$\Pc(\bh)\to\Pc(\bk)$, which we still denote by $\pi$. We define the \demph{Lie kernel} of $\pi$ as the species
\[
\lker(\pi)=\ker\bigl(\pi:\Pc(\bh)\to\Pc(\bk)\bigr).
\]
It is a Lie ideal of $\Pc(\bh)$. The Hopf kernel $\hker(\pi)$ is defined in~\eqref{e:hker}.

\begin{lemma}\label{l:lker}
Let $\pi:\bh\to\bk$ be a morphism of connected Hopf monoids. Then  
\[
\lker(\pi)\subseteq\hker(\pi).
\]
\end{lemma}
\begin{proof} Let $x\in\lker(\pi)$. Then
\[
(\pi_{+\!}\bdot\id)\Delta(x)=(\pi_{+\!}\bdot\id)(1\otimes x + x\otimes 1)=0,
\]
since $\pi_+(1)=0$ and $\pi(x)=0$. Thus $x\in\hker(\pi)$.
\end{proof}

\begin{lemma}\label{l:hker}
Let $\pi:\bh\to\bk$ be a morphism of connected Hopf monoids. Then  $\hker(\pi)$ is a submonoid of $\bh$.
\end{lemma}
\begin{proof} By definition, 
\[
\hker(\pi)=\Delta^{-1}\bigl(\eq(\pi\bdot\id,\, \iota\epsilon\bdot\id)\bigr),
\]
where $\iota:\tone\to\bk$ is the unit of $\bk$, 
$\epsilon:\bh\to\tone$ is the counit of $\bh$, and $\eq$ denotes the equalizer of two maps.
Since $\pi$ and $\iota\epsilon$ are morphisms of monoids $\bh\to\bk$,
the above equalizer is a submonoid of $\bh\bdot\bh$. Since $\Delta$ is a morphism of monoids, $\hker(\pi)$ is a submonoid of $\bh$.
\end{proof}

The following result provides the announced connection between Lie and Hopf kernels
for cocommutative Hopf monoids. It makes use of the Poincar\'e-Birkhoff-Witt and Cartier-Milnor-Moore theorems for species, which are discussed in~\cite[Section~11.9.3]{AguMah:2010}.

\begin{proposition}\label{p:kernels}
Let $\pi:\bh\to\bk$ be a surjective morphism of connected cocommutative Hopf monoids.
Then $\hker(\pi)$ is the submonoid of $\bh$ generated by $\lker(\pi)$.
\end{proposition}
\begin{proof} Lemmas~\ref{l:lker} and~\ref{l:hker} imply one inclusion.
To conclude the equality, it suffices to check that the dimensions agree, or equivalently,
that the exponential generating series are the same. (We are assuming finite dimensionality throughout.)

First of all, from Theorem~\ref{t:maindual}, we have
\[
\exponential{\hker(\pi)}{x}=\exponential{\bh}{x}/\exponential{\bk}{x}.
\]
Now, since $\bh$ is cocommutative, we have
\[
\bh\cong \Uc\bigl(\Pc(\bh)\bigr) \cong \Sc\bigl(\Pc(\bh)\bigr) = \be\circ \Pc(\bh).
\]
The first is an isomorphism of Hopf monoids (the Cartier-Milnor-Moore theorem), the second is an isomorphism of comonoids (the Poincar\'e-Birkhoff-Witt theorem), and the third is the definition of the species underlying $\Sc\bigl(\Pc(\bh)\bigr)$~\cite[Section~11.3]{AguMah:2010}. It follows that
\[
\exponential{\bh}{x} = \exp\bigl(\exponential{\Pc(\bh)}{x}\bigr).
\]
For the same reason, 
\[
\exponential{\bk}{x} = \exp\bigl(\exponential{\Pc(\bk)}{x}\bigr),
\]
and therefore
\[
\exponential{\hker(\pi)}{x}=\exp\bigl(\exponential{\Pc(\bh)}{x} - \exponential{\Pc(\bk)}{x}\bigr).
\]

On the other hand, since the functors $\Uc$ and $\Pc$ define an adjoint equivalence, they preserve
surjectivity of maps. Thus, the induced map $\pi:\Pc(\bh)\to\Pc(\bk)$ is surjective,
and we have an exact sequence
\[
0\to \lker(\pi) \to \Pc(\bh) \to \Pc(\bk) \to 0.
\]
Hence, 
\[
\exponential{\lker(\pi)}{x}=\exponential{\Pc(\bh)}{x} - \exponential{\Pc(\bk)}{x}.
\]
Since $\lker(\pi)$ is a Lie submonoid of $\Pc(\bh)$, the submonoid of $\bh$ generated by $\lker(\pi)$ identifies with $\Uc\bigl(\lker(\pi)\bigr)$. Therefore, as above, the generating series for the latter submonoid is
\[
\exp\bigl(\exponential{\lker(\pi)}{x}\bigr)=
\exp\bigl(\exponential{\Pc(\bh)}{x} - \exponential{\Pc(\bk)}{x}\bigr)=\exponential{\hker(\pi)}{x},
\]
which is the desired equality.
\end{proof}

\begin{remark} The results of this section hold also for connected (not necessarily graded or finite dimensional) Hopf algebras.
See~\cite[Example 4.20]{BCM:1986} for a proof of Proposition~\ref{p:kernels} in this setting. The proof above used finite dimensionality of the (components of the) species, but this hypothesis is not necessary.
The proof in~\cite{BCM:1986} may be adapted to yield the result
for arbitrary species.
\end{remark}

\subsection{The Lie kernel of $\pi:\bl\onto\be$}\label{ss:cycle}

We return to the discussion in Example~\ref{eg:derangement}.
The primitive elements of the Hopf monoids $\be$ and $\bl$ are
described in~\cite[Example~11.44]{AguMah:2010}. We have that
\[
\Pc(\be)=\bX \qand \Pc(\bl)=\bLie
\]
where $\bX$ is the species of \demph{singletons},
\[
\bX[I] =
\begin{cases}
\field & \text{ if $\abs{I}=1$, } \\
   0      & \text{ otherwise,}
\end{cases}
\]
 and $\bLie$ is the species underlying the \demph{Lie operad}.
It follows that the Lie kernel of the canonical morphism $\pi:\bl\onto\be$ is given by
\begin{equation}\label{e:lker}
\lker(\pi)[I]=
\begin{cases}
\bLie[I] & \text{ if $\abs{I}\geq 2$, } \\
   0      & \text{ otherwise.}
\end{cases}
\end{equation}

Before moving on to the Hopf kernel of $\pi$, we provide some more information on the species $\bLie$. 

Let $I$ be a finite nonempty set and $n=\abs{I}$.
It is known that the space $\bLie[I]$ is of dimension $(n-1)!$. We proceed to describe a linear basis indexed by \demph{cyclic orders} on $I$. A cyclic order on $I$ is an equivalence class of linear orders on $I$ modulo the action 
\[
i_1\cmrg \cdots \cmrg i_{n-1} \cmrg i_n \ \mapsto \ i_n\cmrg i_1\cmrg \cdots \cmrg i_{n-1}
\]
of the cyclic group of order $n$. Each class has $n$ elements so there are $(n-1)!$ cyclic orders on $I$.
We use $(b,a,c)$ to denote the equivalence class of the linear order $b\cmrg a\cmrg c$.

We fix a finite nonempty set $I$ and choose a linear order $\ell_0$ on $I$, say
\[
\ell_0 = i_1\cmrg i_2 \cmrg \cdots \cmrg i_n.
\]
The basis of $\bLie[I]$ will depend on this choice.
Given a cyclic order $\gamma$ on $I$, let $S$ be the subset of $I$ consisting of
the elements encountered when traversing the cycle from $i_1$ to $i_2$ clockwise,
including $i_1$ but excluding $i_2$ (these are the first and second elements in $\ell_0$, respectively). Let $T$ consist of the remaining elements (from $i_2$ to $i_1$). Note that $i_1\in S$ and $i_2\in T$, so both $S$ and $T$ are nonempty.
The cyclic order $\gamma$ on $I$ induces cyclic orders on $S$ and $T$. We denote them by $\gamma|_S$ and $\gamma|_T$.  An element $p_\gamma\in \bl[I]$ is defined recursively by
\[
p_\gamma := [p_{\gamma|_S}, p_{\gamma|_T}]= p_{\gamma|_S}\cdot p_{\gamma|_T}-
p_{\gamma|_T}\cdot p_{\gamma|_S}.
\]
The elements $p_{\gamma|_S}\in\bl[S]$ and $p_{\gamma|_T}\in\bl[T]$ are themselves defined
with respect to the induced linear orders $(\ell_0)|_S$ and $(\ell_0)|_T$. The recursion starts with
the case when $I$ is a singleton $\{a\}$. In this case, we set
\[
p_{(a)}:=a\in \bl[a]
\]
(the unique linear order).

Clearly $a\in\bl[a]$ is a primitive element. Since the primitive elements are closed under commutators, we have $p_\gamma\in \bLie[I]$. Moreover, we have the following.

\begin{proposition}\label{p:liebase}
For fixed $I$ and $\ell_0$ as above, the set
\[
\bigl\{p_\gamma \mid \text{$\gamma$ is a cyclic order on $I$}\bigr\}
\]
is a linear basis of $\bLie[I]$.
\end{proposition}
\begin{proof}
The construction of the elements $p_\gamma$ is a reformulation of the familiar construction of the \emph{Lyndon}
basis of a free Lie algebra~\cite{Lot:1997,Reu:1993,Reu:2003}. Reading the elements of the cyclic order $\gamma$ clockwise starting
at the minimum of $\ell_0$ gives rise to a Lyndon word on $I$ (without repeated letters).
The cyclic orders $\gamma|_S$ and $\gamma|_T$ give rise to the Lyndon words in the canonical factorization of this Lyndon word. 
\end{proof}

For example, suppose that $I=\{a,b,c,d\}$, $\ell_0=a\cmrg b\cmrg c\cmrg d$ and $\gamma=(b,a,c,d)$. Then
\begin{align*}
p_{(b,a,c,d)} &= [p_{(a,c,d)}, p_{(b)}]=\bigl[ [p_{(a)}, p_{(c,d)}], p_{(b)}\bigr]
=\Bigl[ \bigl[ p_{(a)}, [p_{(c)}, p_{(d)}] \bigr], p_{(b)} \Bigr]
=\Bigl[ \bigl[ a, [c, d] \bigr], b \Bigr] \\
&=
a\cmrg c\cmrg d\cmrg b - a\cmrg d\cmrg c\cmrg b - c\cmrg d\cmrg a\cmrg b
+d\cmrg c\cmrg a\cmrg b
-b\cmrg a\cmrg c\cmrg d +b\cmrg a\cmrg d\cmrg c +b\cmrg c\cmrg d\cmrg a
-b\cmrg d\cmrg c\cmrg a.
\end{align*}

\begin{remark} 
The vector species $\bLie$ is \emph{not} the linearization of the set species of cycles.
Note also that, for a general bijection $\sigma:I\to J$, the $p$-basis of $\bLie[I]$
will not map to the $p$-basis of $\bLie[J]$ under $\bl[\sigma]$. 
\end{remark}

\subsection{The Hopf kernel of $\pi:\bl\onto\be$}\label{ss:derangement}

The above description~\eqref{e:lker} of the Lie kernel of $\pi:\bl\onto\be$ together with
Proposition~\ref{p:kernels} imply that the Hopf kernel of $\pi$ is given by
\[
\hker(\pi)[I]=\sum_{k\geq 1}\sum_{\substack{S_1\sqcup\cdots\sqcup S_k=I \\ \abs{S_r}\geq 2\,\,\forall r}}
\bLie[S_1]\cdots\bLie[S_k].
\]
An element in $\bLie[S_1]\cdots\bLie[S_k]$ is a $k$-fold product of primitive elements $x_r\in\bLie[S_r]$; each $S_r$ must have at least $2$ elements. We proceed to describe
a linear basis for $\hker(\pi)[I]$.

As in Section~\ref{ss:cycle}, we fix a linear order $\ell_0=i_1\cmrg i_2 \cmrg \cdots \cmrg i_n$ on $I$. The basis will be indexed by \demph{derangements} of $\ell_0$.
A derangement of $\ell_0$ is a linear
order $\ell=j_1\cmrg j_2 \cmrg \cdots \cmrg j_n$ on $I$ such that $i_r\neq j_r$ for all $r=1,\ldots,n$. 

View linear orders as bijections $[n]\to I$ and define $\sigma:=\ell\circ\ell_0^{-1}$.
Then $\sigma$ is a permutation of $I$ and $\ell$ is a derangement of $\ell_0$ precisely when $\sigma$ has no fixed points. 

Let $\ell$ be a derangement of $\ell_0$ and $\sigma$ the associated permutation. 
Let $S_1,\ldots,S_k$ be the orbits of $\sigma$ on $I$ labeled so that
\[
\min S_1<\cdots<\min S_k 
\text{ \ according to $\ell_0$,}
\]
and let $\gamma_r$ be the cyclic order on $S_r$ induced by $\sigma$. In other words,
$\sigma=\gamma_1\cdots\gamma_k$
is the factorization of $\sigma$ into cycles, ordered in this specific manner.

Employing the $p$-basis of $\bLie$ from Section~\ref{ss:cycle} (defined with respect to $\ell_0$ and the orders induced by $\ell_0$ on subsets of $I$), we define an element $p_{\ell}\in\bl[I]$ by
\[
p_{\ell} :=p_{\gamma_1}\cdots p_{\gamma_k}. 
\]
By assumption, $\abs{S_r}\geq 2$ for all $r$.
Hence $p_{\gamma_r}\in\lker(\pi)[S_r]$ and $p_{\ell}\in\hker(\pi)[I]$. 

For example, let $I=\{e,i,m,s,t\}$, $\ell_0 =s\cmrg m \cmrg i \cmrg t \cmrg e$ and $\ell = i \cmrg t \cmrg e \cmrg m \cmrg s$. Then 
\[
\sigma = (s,i,e)(m,t), \quad S_1 = \{i,e,s\},\quad S_2 = \{m,t\}, 
\]
and
\[
p_\ell = p_{(s,i,e)}p_{(m,t)} = \left[p_{(s)},p_{(i,e)}\right]p_{(m,t)} = \bigl[ s,[i,e]\bigr][m,t].
\]

\begin{proposition}\label{p:hopfbase}
For fixed $I$ and $\ell_0$ as above, the set
\[
\bigl\{p_\ell \mid \text{$\ell$ is a derangement of $\ell_0$}\bigr\}
\]
is a linear basis of $\hker(\pi)[I]$.
\end{proposition}
\begin{proof}
This follows from Proposition~\ref{p:liebase} and the Poincar\'e-Birkhoff-Witt theorem.
\end{proof}

\begin{example} We describe the $p$-basis of $\hker(\pi)[I]$ in low cardinalities.
Throughout, we choose
\[
\ell_0=a\cmrg b \cmrg c \cmrg\cdots\,.
\]
The space $\hker(\pi)[a,b]$ is $1$-dimensional, linearly spanned by 
\[
p_{b\cmrg a} = p_{(a,b)} =[a,b].
\]
The space $\hker(\pi)[a,b,c]$ is $2$-dimensional, linearly spanned by 
\begin{align*}
p_{b\cmrg c\cmrg a} = p_{(a,b,c)} = [ p_{(a)},p_{(b,c)}] = \bigl[ a,[b,c]\bigr],\\
p_{c\cmrg a\cmrg b} = p_{(a,c,b)} = [p_{(a,c)},p_{(b)}] = \bigl[ [a,c], b\bigr].
\end{align*}
The space $\hker(\pi)[a,b,c,d]$ is $9$-dimensional. There are $6$ basis elements corresponding to $4$-cycles, such as
\[
p_{c\cmrg a\cmrg d\cmrg b} = p_{(a,c,d,b)} = \Bigl[ \bigl[ a, [c, d] \bigr], b \Bigr],
\]
and $3$ basis elements corresponding to products of two $2$-cycles, such as
\[
p_{b\cmrg a\cmrg d\cmrg c} = p_{(a,b)}p_{(c,d)} = [a,b]\cdot [c,d].
\]
\end{example}


\section*{Appendix}

The following fact was referred to in the last remark in Section~\ref{s:dimensions}.

\begin{proposition}\label{p:domain}
Let $H$ be a connected (not necessarily graded) Hopf algebra over a field of characteristic $0$.
Then $H$ is a domain.
\end{proposition}
This result is proven in~\cite[Lemma~1.8(a)]{WZZ:2011}, where it is attributed to Le Bruyn.
We provide a different proof here.
\begin{proof}
Let $K$ denote the associated graded Hopf algebra with respect to the coradical
filtration of $H$. Since $H$ is connected, $K$ is commutative~\cite[Remark~1.7]{AguSot:2005}. Now by~\cite[Proposition~1.2.3]{Rad:1979}, $K$ embeds in a shuffle Hopf algebra. The latter is a free commutative algebra~\cite[Corollary~3.1.2]{Rad:1979}, hence a
domain. It follows that $K$ and hence also $H$ are domains.
\end{proof}
 

Over a field of positive characteristic, the restricted universal enveloping algebra $\mathfrak u(\mathfrak g)$ of a finite dimensional nonzero Lie algebra $\mathfrak g$ 
is a connected Hopf algebra that is not a domain. Indeed, in this case $u(\mathfrak g)$ is finite dimensional~\cite[Theorem~V.12]{Jac:1962} and so has a nontrivial idempotent, being a Hopf algebra with integrals~\cite[Theorem~2.1.3]{Mon:1993}.


\bibliographystyle{plain}  
\bibliography{bibl}

\end{document}